\documentclass[12pt]{amsart}

  \usepackage{a4wide}
  \usepackage{amssymb}
  \usepackage{amsmath}
  \usepackage{amsfonts}
  \usepackage{amscd}

  \usepackage{graphicx}
  \usepackage{boxedminipage}
  \usepackage[format=hang,font=small]{caption} 

  \usepackage{color}

  \def\KK{\mathbb{K}}
  \def\FF{\mathbb{F}}
  \def\ZZ{\mathbb{Z}}
  
  \def\Cc{\mathcal{C}}

  \def\SL{\mathrm{SL}}
  \def\GL{\mathrm{GL}}
  \def\Fix{\mathrm{Fix}}

  \makeatletter
  \CheckCommand*\refstepcounter[1]{\stepcounter{#1}%
      \protected@edef\@currentlabel
       {\csname p@#1\endcsname\csname the#1\endcsname}%
  }
  \renewcommand*\refstepcounter[1]{\stepcounter{#1}%
    \protected@edef\@currentlabel
      {\csname p@#1\expandafter\endcsname\csname the#1\endcsname}%
  }
  \def\labelformat#1{\expandafter\def\csname p@#1\endcsname##1}
  \DeclareRobustCommand\Ref[1]{\protected@edef\@tempa{\ref{#1}}%
     \expandafter\MakeUppercase\@tempa
  }
  \makeatother

  \makeatletter
  \newcommand{\numberlike}[2]{%
     \expandafter\def\csname c@#1\endcsname{%
         \expandafter\csname c@#2\endcsname}%
  }
  \makeatother

  \def\DefaultNumberTheoremWithin{section}

  \theoremstyle{plain}
  \newtheorem{Lemma}{Lemma}
     \numberwithin{Lemma}{\DefaultNumberTheoremWithin}
     \labelformat{Lemma}{Lemma~#1}
  
     \numberwithin{Claim}{\DefaultNumberTheoremWithin}
     \numberlike{Claim}{Lemma}
     \labelformat{Claim}{Claim~#1}
  \newtheorem{Theorem}{Theorem}
     \numberwithin{Theorem}{\DefaultNumberTheoremWithin}
     \numberlike{Theorem}{Lemma}
     \labelformat{Theorem}{Theorem~#1}
  \newtheorem{Corollary}{Corollary}
     \numberwithin{Corollary}{\DefaultNumberTheoremWithin}
     \numberlike{Corollary}{Lemma}
     \labelformat{Corollary}{Corollary~#1}
  \newtheorem{Proposition}{Proposition}
     \numberwithin{Proposition}{\DefaultNumberTheoremWithin}
     \numberlike{Proposition}{Lemma}
     \labelformat{Proposition}{Proposition~#1}
  
     \numberwithin{Conjecture}{\DefaultNumberTheoremWithin}
     \numberlike{Conjecture}{Lemma}
     \labelformat{Conjecture}{Conjecture~#1}

  \theoremstyle{definition}
  
     \numberwithin{Definition}{\DefaultNumberTheoremWithin}
     \numberlike{Definition}{Lemma}
     \labelformat{Definition}{Definition~#1}

  \theoremstyle{definition}
  \newtheorem{Question}{Question}
     \numberwithin{Question}{\DefaultNumberTheoremWithin}
     \numberlike{Question}{Lemma}
     \labelformat{Question}{Question~#1}

  \theoremstyle{definition}
  
     \numberwithin{Problem}{\DefaultNumberTheoremWithin}
     \numberlike{Problem}{Lemma}
     \labelformat{Problem}{Problem~#1}

  \theoremstyle{remark}
  
     \numberwithin{Remark}{\DefaultNumberTheoremWithin}
     \numberlike{Remark}{Lemma}
     \labelformat{Remark}{Remark~#1}
  \theoremstyle{remark}

  \newtheorem{Example}{Example}
     \numberwithin{Example}{\DefaultNumberTheoremWithin}
     \numberlike{Example}{Lemma}
     \labelformat{Example}{Example~#1}
  
     \labelformat{Case}{Case~#1}
     \numberwithin{Case}{Lemma}
  
     \labelformat{Step}{Step~#1}
     \numberwithin{Step}{Lemma}

  \labelformat{figure}{Figure~#1}
  \labelformat{equation}{(#1)}
  \labelformat{section}{\S#1}
  \labelformat{subsection}{\S#1}
  \labelformat{subsubsection}{\S#1}

  \newenvironment{note}[1][Note]
   {\bigskip\begin{center}\begin{boxedminipage}{4.5in}\setlength{\parindent}{1em}\noindent\textbf{#1. }}
   {\end{boxedminipage}\end{center}\bigskip}

\def\srl{{\mathcal R}}
\def\cl{{\mathrm{cl}}}

\def\la{\langle}
\def\ra{\rangle}

\def\rk{{{\mathsf r}{\mathsf k}}}
\def\il{{\mathsf I}{\mathsf n}{\mathsf t}}
\def\ov{\overline}
\def\Core{\mathrm{Core}}
\def\nnm{{\mathcal M}}

\begin{document}

\title{On the lattice of subracks of the rack of 
       a finite group}

\author{Istvan Heckenberger}
\address{Philipps-Universit\"at Marburg, 
         Fachbereich Mathematik und Informatik,
         35032 Marburg, Germany}
\email{heckenberger@mathematik.uni-marburg.de}

\author{John Shareshian}
\address{
    Department of Mathematics,
    Washington University,
    St. Louis, MO 63130, USA}
\email{shareshi@math.wustl.edu}

\author{Volkmar Welker}
\address{Philipps-Universit\"at Marburg, 
         Fachbereich Mathematik und Informatik,
         35032 Marburg, Germany}
\email{welker@mathematik.uni-marburg.de}

\begin{abstract}
  In this paper we initiate the study of racks from the combined perspective of
  combinatorics and finite group theory. A rack $R$ is a set with a self-distributive
  binary operation. 
  We study the combinatorics of the partially ordered set $\srl(R)$ 
  of all subracks of $R$ with inclusion as the order relation. 
  Groups $G$ with the conjugation operation provide an important class
  of racks. For the case $R = G$ we show that  
  \begin{itemize}
     \item the order complex of $\srl(R)$ has the homotopy type of a sphere,
     \item the isomorphism type of $\srl(R)$ determines if $G$ is abelian,
           nilpotent, supersolvable, solvable or simple,
     \item $\srl(R)$ is graded if and only if $G$ is abelian, $G = S_3$, $G = D_8$ or $G = Q_8$.
  \end{itemize}
  In addition, we provide some examples of subracks $R$ of a group $G$
  for which $\srl(R)$ relates to well studied combinatorial structures.
  In particular, the examples show that the order complex of $\srl(R)$ for 
  general $R$ is more complicated than in the case $R = G$.
\end{abstract}

\maketitle

\section{Introduction}

   A {\it rack} $R$ is a set (possibly empty) together with a binary 
   operation $\triangleright$ satisfying the following properties
   \begin{itemize}
     \item {\sf (Self Distributivity)} for all $a,b,c \in R$ we have 
        $a \triangleright (b \triangleright c) = (a \triangleright b) 
                    \triangleright (a \triangleright c)$
     \item {\sf (Bijectivity)} for all $a,b \in R$ there is a unique 
        $c \in R$ such that $a \triangleright c = b$. 
   \end{itemize}

   Rack like structures first appeared in finite geometry, in an
   attempt to axiomatize the concept of a reflection \cite{Th}. In 
   the 1980s racks and subclasses of racks emerged in knot theory under various 
   names (e.g. quandle (Joyce), distributive groupoid
   (Matveev), automorphic set (Brieskorn)). We refer the reader to \cite{ElNe} 
   for more details.  

%   In \cite{Jo}, the idea to
%   describe a (topological) knot using a (algebraic) rack appeared 
%   and was worked out in detail. On the way, some basic
%   structure theory of racks was developed and various groups attached 
%   to a rack were studied. In \cite{Br}, automorphic sets 
%   (that is, racks) were discussed in the context of Riemannian 
%   symmetric spaces, vector spaces with bilinear forms, and root systems. 
%   Particular attention is payed to relationships with braid group actions 
%   and to monodromy groups of singularities.
	 
   Racks are intrinsically related to braidings. For that reason, in 
   recent years they have been scrutinized in the theory of braided vector 
   spaces and associated Hopf algebras \cite{AnGr}, \cite[Sect.~3.6]{An}.
   In this context a (co)homology theory of racks was developed and 
   (co)homological properties of racks were studied \cite{EtGr}.
   Other structural invariants of a rack are its inner group and its enveloping group.
   Rack morphisms can be used to define and classify (finite) simple racks. 
   Finite racks of small size were classified in \cite{Ve,Mc,HuStVo}.
	 
   In this paper, we initiate the
   study of (finite) racks from the combined perspective of combinatorics and 
   finite group theory.

%   For any rack, the set theoretic inclusion gives the set of subracks the 
%   structure of a partially ordered set, poset for short. 
%   We are particularly interested in the case where 
%   the rack is a finite group endowed with the conjugation action. We are 
%   going to discuss to which extent the group can be recovered from its 
%   poset of subracks, show that there is a topological restriction on this 
%   poset, and exhibit restrictions on the group
%   structure implied by combinatorial conditions imposed on the poset. 
%   This parallels work on subgroup
%   posets of groups and posets of cosets of subgroups (see \cite{Smith},
%   \cite{ShareshianWoodroofe}). 
%   Since one can identify the sizes of conjugacy classes of a 
%   group from its poset of subracks by \ref{lem:avoid} our work can 
%   also be seen as a new approach
%   to studying the influence of conjugacy class sizes on the structure of 
%   a group (see \cite{Ca,Ma}).

    The set of subracks of a rack $R$ is partially ordered by inclusion.  
    We will study the structure of this poset in the case that $R$ is a finite group and $a \triangleright b=aba^{-1}$.  
    We will discuss the extent to which the algebraic structure of a group is determined by the combinatorial structure 
    of its subrack poset and determine the topological structure of the order complex of this poset.  Similar work using 
    subgroup lattices, coset posets and posets of $p$-subgroups has appeared in various papers and books, including 
    \cite{Brown,Q,Schmidt,Smith,ShareshianWoodroofe,ShareshianWoodroofeII,SuzukiII}.
    We show in \ref{lem:avoid} that the combinatorial structure of the subrack lattice of a finite group $G$ determines the 
    multiset of conjugacy class sizes of $G$.  Thus our work is related to the study of the influence on this multiset 
    on the structure of $G$ (see \cite{Ca,Ma}).

   We proceed by providing precise definitions and statements of our main results.
   If $G$ is a group and $R$ is a union of conjugacy classes of $G$ then $R$ 
   together with
   $a \triangleright b := aba^{-1}$ is easily checked to be a rack.
   This rack satisfies the additional identity
   \begin{itemize}
      \item for all $a \in A$ we have $a \triangleright a = a$.
   \end{itemize}
   Racks satisfying this identity are sometimes called {\it quandles}.
   A subset $Q \subseteq R$ of a rack is called a subrack if $Q$ with 
   the operation $\triangleright$ from $R$ is a rack. 
   Note that $R$ is a quandle if and only if every singleton is a subrack.

   We write $\srl(R)$ for the set of all subracks of $R$, partially ordered by containment.
   \ref{le:lattice} states the simple fact that $\srl(R)$ is indeed a
   lattice. 
   If $R$ is the rack of elements of a group $G$ with 
   $a \triangleright b := aba^{-1}$
   then we write $\srl(G)$ for the partially ordered set of subracks of this rack.

   We usually write $\leq$ for the order relation in $\srl(R)$.
   
%   By $\Delta(\srl(R))$ we denote the order complex of $\srl(R)$; that is
%   the simplicial complex of all linearly ordered subsets of  
%   $\srl(R) \setminus \{\emptyset,R\}$.
  
   We write $\Delta(\srl(R))$ for the order complex of $\srl(R)$.  So, $\Delta(\srl(R))$ is the abstract 
   simplicial complex whose faces are all linearly ordered subsets of $\srl(R) \setminus \{\emptyset,R\}$.
   
   We show that up to homotopy $\Delta(\srl(G))$ is determined by the 
   number of conjugacy classes of $G$.

   \begin{Proposition} 
      \label{pr:grouprack}
      Let $G$ be a finite group with $c$ conjugacy classes of elements
      and $z(G)$ elements in the center of $G$. 
      If  $R$ is the rack of all non-central conjugacy classes in $G$ and 
      $Z$ is any set of central elements in $G$, then 
      \begin{enumerate} 
         \item $\srl(R \cup Z) \cong \srl(R) \times 2^Z$, and 
         \item $\Delta(\srl(R \cup Z))$ is homotopy equivalent to a $(c-z(G)+|Z|-2)$-sphere. 
      \end{enumerate}
      In particular, $\Delta(\srl(G))$ is homotopy equivalent to a 
      $(c-2)$-sphere.
   \end{Proposition}

   As a first structural consequence, the combinatorics of $\srl(G)$ 
   determines the number of conjugacy classes of $G$. In our main result we 
   show that indeed more group theoretical information is encoded in $\srl(G)$.

   \begin{Theorem} \label{main}
    Let $G,H$ be finite groups satisfying $\srl(G) \cong \srl(H)$.  
    \begin{enumerate} 
      \item If $G$ is abelian then $H$ is abelian. 
      \item If $G$ is nilpotent then $H$ is nilpotent. 
      \item If $G$ is supersolvable then $H$ is supersolvable. 
      \item If $G$ is solvable then $H$ is solvable. 
      \item If $G$ is simple then $H$ is simple. 
    \end{enumerate}
   \end{Theorem}
 
   Finally, we are able to classify the finite groups $G$ for which 
   $\srl(G)$ is graded; that is all maximal chains have the same length.

   \begin{Theorem} \label{graded}
     Let $G$ be a finite group.  The lattice $\srl(G)$ is graded if and only if 
     $G$ is abelian or $G$ is isomorphic to one of $S_3$, $D_8$ or $Q_8$.
   \end{Theorem}

   The paper is organized a follows. In Section \ref{sec2} we provide
   the basic notation and basic properties of the poset $\srl(R)$ of subracks
   of the rack $R$. In addition, we provide the proof of \ref{pr:grouprack} and exhibit
   examples that show that in general $\srl(R)$ need not be Cohen-Macaulay, or even graded,
   and that $\Delta(\srl(R))$ can have the homotopy type of a wedge of spheres of different dimensions.
   In \ref{sec3} we provide the proof of 
   \ref{main}.  The proof of \ref{graded} appears in \ref{sec4}.
   We close with a few open questions in \ref{sec5}.

\section{Racks and Subracks}
   \label{sec2}

   We know of no general restriction on the structure of the poset
   $\srl(R)$ other than the following one (see \ref{sec5} for known 
   restrictions in case $R$ is a quandle).

   \begin{Lemma}
     \label{le:lattice}
      If $R$ is a rack, then the partially ordered set $\srl(R)$ is a lattice.
      If $R$ in addition is a quandle. then $\srl(G)$ is atomic.
   \end{Lemma}
   \begin{proof}
     Clearly, $\emptyset$, $R \in \srl(R)$ are the unique minimal and maximal 
     elements of $\srl(R)$.
     If $Q, Q'$ are two subracks of $R$ then for $a,b \in Q \cap Q'$ we 
     have $a \triangleright b \in Q \cap Q'$. As $\triangleright$ is 
     self-distributive on $Q$ and $Q'$ it is so on $Q \cap Q'$. 
     Similarly for $a,b \in Q \cap Q'$ there is a unique $c \in R$ such that 
     $a \triangleright c = b$. Since $Q$ and $Q'$ are subracks this $c$
     must then lie in $Q \cap Q'$. Hence $Q \cap Q'$ is a subrack. 
     But then $Q \cap Q'$ is the unique maximal subrack contained in $Q$ 
     and $Q'$ thus the infimum of $Q$ and $Q'$ exists. Since $\srl(R)$ has a 
     unique maximal element this implies that $Q$ and $Q'$ also
     have a supremum and hence that $\srl(R)$ is a lattice.
   \end{proof}

   If $R$ is a quandle then each subrack of $R$ is also a quandle.
   Therefore, there is no need to consider the poset of subquandles of a 
   quandle separately. 
 
   Next we turn to examples where $R$ is a subrack of the rack given
   by a group and conjugation.
   The following fact will be useful and is an immediate consequence 
   of the definitions.

   \begin{Lemma}
     \label{le:groupcase}
     Let $R = G$ be the rack given by a group $G$ and conjugation.
     If $Q$ is a subrack of $R$ such that $G = \langle Q \rangle$ 
     is the group generated by $Q$, then $Q$ is a union of conjugacy 
     classes of $G$. 
   \end{Lemma}

   Next we show that familiar combinatorial objects arise among 
   subrack lattices $\srl(R)$.

   \begin{Example}
     \label{ex:transposition}
      Let $G = S_n$ be the symmetric group on $n$ letters and let $R$ be the 
      quandle, defined by the conjugacy class of transpositions.
      Let $Q$ be a subrack of $R$. Let $H = \langle Q \rangle$ then $H$ is a 
      subgroup of $S_n$ generated by transpositions. Hence 
      $H = S_{B_1} \times \cdots \times S_{B_r}$
      is the direct product of full symmetric groups 
      $S_{B_i}$ on the blocks $B_i$ of a set partition 
      $B_1|\cdots|B_r$ of $[n]:=\{1,\ldots, n\}$.
      By \ref{le:groupcase} it follows that $Q$ is a union of conjugacy 
      classes of $H$.
      Since $Q$ consists of transpositions it follows that $Q$ is the set of all
      transpositions $\tau = (\ell~k)$ for which there is some $i$ such that 
      $\{ \ell,k\} \subseteq B_i$. 
      Thus there is an isomorphism between the lattice $\srl(R)$ and the 
      lattice $\Pi_n$ of set partitions of $[n]$.
   \end{Example}

   It is interesting to see the partition lattice appear among the $\srl(R)$,
   but it is somewhat misleading as in general we do not know what restrictions
   if any govern the structure of $\srl(R)$. 
   
   One possible approach to capture the structure of partially ordered sets is
   the study of their order complexes. 
   For a partially ordered set $P$ with unique maximal element 
   $\hat{1}$ and unique minimal element $\hat{0}$ the {\it order complex} 
   $\Delta(P) = \{ \hat{0} < p_1 < \cdots < p_\ell < \hat{1}~|~p_i \in P\}$ 
   of $P$ is the simplicial complex of all linearly ordered 
   subsets of $P \setminus \{ \hat{0},\hat{1}\}$.
   In particular, we can now speak of homology and homotopy type when we 
   consider partially ordered sets.
   An important homological property of partially ordered sets is the
   {\it Cohen-Macaulay} property (over a field $\KK$); see \cite[11.5]{Bjoerner}. 

   \begin{Example}
     The order complex $\Delta(\Pi_n)$ of the partition lattice is a geometric 
     lattice and hence well known to be Cohen-Macaulay over any field (see \cite[11.10]{Bjoerner}).
     Thus by \ref{ex:transposition} the subrack lattice $\srl(R)$ for $R$
     the rack of transpositions in 
     $S_n$ is Cohen-Macaulay over any field. 
   \end{Example}

   One consequence of a poset $P$ being Cohen-Macaulay over a field $\KK$ is the 
   property that the homology of $\Delta(P)$ is concentrated in dimension $r-2$, where
   $r$ is the length of a maximal chain in $P$.  
   This fails for $\srl(R)$ in general.

   \begin{Example}
      \label{ex:fourcycle}
      Let $R$ be the rack of $4$-cycles in $S_4$. 
      Then $R$ contains $6$ elements. Clearly, every $4$-cycle is a subrack
      and any pair consisting of a $4$-cycle and its inverse is a subrack. But if we have
      two $4$-cycles that are not inverses of each other then they generate
      $S_4$ as no subgroup of $S_4$ contains more than one cyclic subgroup of
      order $4$.
      Thus $\srl(R) \setminus \{ R,\emptyset\}$ is disconnected and 
      $\dim \Delta(\srl(R))$ is $1$. Thus $\Delta(\srl(R))$ is not 
      Cohen-Macaulay.
   \end{Example}

   Another consequence of the Cohen-Macaulay property for a partially ordered
   set $P$ is that $P$ is graded.

   \begin{Example}
     \label{ex:chains}
     Let $R$ be the rack of all $5$-cycles in $A_5$.  Then $R$ splits into 
     two conjugacy classes, $C$ and $D$, as a $5$-cycle $x$ generates its 
     own centralizer in $S_5$. In particular, a $5$-cycle $x$ is conjugate 
     to $x^4$ but to neither $x^2$ nor $x^3$. Also, if a $5$-cycle $y$ is 
     not a power of $x$ then $ \langle x,y \rangle =A_5$.

     For $x \in C$ we have that 

     $$\emptyset < \{x\} < \{x,x^2\} < \{x,x^2,x^3\} < \{x,x^2,x^3,x^4\} < R$$
 
     is a maximal chain of length $5$.

     On the other hand,

     $$\emptyset < \{x\} < \{x,x^4\} < C < R$$

     is another maximal chain and of length $3$.
     Thus $\srl(R)$ is not graded and hence not Cohen-Macaulay.
   \end{Example}

   In the two examples above we have studied subracks $R$ of the 
   rack of a finite group $G$ with conjugation.
   Hence the subrack lattices $\srl(R)$ appear as
   intervals in $\srl(G)$. This may indicate that $\srl(G)$ can
   exhibit an almost arbitrary behavior. Next we show that this is
   not the case.
   As a preparation for the proof we need the following well known 
   fact from group theory.

   \begin{Lemma}
    \label{lem:avoid} 
    Let $G$ be a group and $H < G$ a proper subgroup of $G$. 
    Then there is a conjugacy class $C$ in $G$ such that $C \cap H =
    \emptyset$.  
  \end{Lemma}
  \begin{proof}
    Consider the set $X = \{ Hg ~|~g\in G\}$ of right cosets of $H$ in $G$ as a $G$-set
    with $G$ acting by right multiplication.
    The Burnside lemma states
    $$ | X/G| = \frac{1}{|G|} \sum_{g \in G} |\Fix_g(X)|$$ where
    $X/G$ is the set of orbits of $G$ on $X$ and 
    $\Fix_g(X) = \{ x \in X~ |~ xg = x\}$. Clearly, $|X/G| = 1$ and
    $|\Fix_1(X)| = |X| > 1$.

    Assume for contradiction that $H$ contains an element of every conjugacy class.
    Then $\bigcup_{h \in G} H^h = G$. 
    For any $g \in G$, we have $Hh g = Hh$ if and only if 
    $g \in H^{h}$. Thus $|\Fix_g(X)| \geq 1$
    for every $g \in G$.
    Now 
    $$ 1 = | X/G| = \frac{1}{|G|} \sum_{g \in G} |\Fix_g(X)| > 1,$$
    yielding the desired contradiction.
  \end{proof}
   
  Now we are in position to prove the crucial lemmas leading 
  to \ref{pr:grouprack}.
  
  \begin{Lemma}
    \label{le:maxgrrack}
    Let $G$ be a finite group.
    A subrack $Q$ of $G$ is maximal in $\srl(G)$ if and only if
    $Q$ is the union of all but one conjugacy class of $G$.
  \end{Lemma}
  \begin{proof}
     Let $Q \in \srl(G)$ be a subrack of $G$. 
     If there is no conjugacy 
     class $C$ of $G$ for which $C \cap Q = \emptyset$ then by 
     \ref{lem:avoid} we must have $\langle Q \rangle = G$. 
     Then by \ref{le:groupcase} $Q$ must be a union of conjugacy classes 
     and hence $Q = G$. 
     Thus if $Q \neq G$ then there is a conjugacy class $C$ of $G$ with 
     $C \cap Q = \emptyset$. Since a union of conjugacy classes is 
     always a subrack, it follows that the maximal
     elements of $\srl(G)$ are the unions of all but one conjugacy
     class in $G$.
  \end{proof}

  \begin{Lemma}
    \label{le:product}
    Let $G$ be a finite group and $R$ be a subrack of $G$ that is a union of
    conjugacy classes.
    Then:
    \begin{itemize} 
      \item[(1)] 
        If $Z$ is a set of central elements such that 
        $Z \cap R = \emptyset$ then the map $Q \mapsto (Q \cap R,Q \cap Z)$ 
        induces an isomorphism $\srl(R \cup Z) \cong \srl(R) \times \srl(Z)$.
      \item[(2)] If $Q$ is a maximal element of $\srl(R)$ then
        $Q$ is the union of all but one conjugacy class from $R$.
    \end{itemize}
  \end{Lemma}
  \begin{proof}
     \begin{itemize}
       \item[(1)] The claim is easily verified. 
       \item[(2)] For $R= G$, $Z = \emptyset$ the claim holds by
         \ref{le:maxgrrack}.
      
         Next we consider the general case. For that, let $Z$ be the set of
         central elements of $G$ not contained in $R$. Then by (1) 
         there is an isomorphism from $\srl(G)$ to $\srl(R) \times \srl(Z)$, 
         sending $Q \in \srl(G)$ to $(Q \cap R, Q \cap Z)$. Since we 
         already know that each maximal element of $\srl(G)$ is obtained from $G$
         by removing a conjugacy class, it follows that the maximal elements
         of $\srl(R)$ are the subsets that are the union of all
         but one conjugacy class in $R$.
    \end{itemize}
  \end{proof}

   \begin{proof}[Proof of \ref{pr:grouprack}]
     Claim (i) follows from \ref{le:product} (1) and the fact
     that $\srl(Z)=2^Z$ whenever $Z \subseteq G$ is a set of pairwise commuting
     elements.

     For (ii) consider the map $\phi : \srl(R \cup Z) \rightarrow \srl(R \cup Z)$ that
     sends a subrack $Q$ to the union of all conjugacy classes $C$ for which 
     $Q \cap C \neq \emptyset$. Then $Q \leq \phi(Q) = \phi(\phi(Q))$ 
     and $Q \leq Q'$ implies $\phi(Q) \leq \phi(Q')$. By \ref{le:product}(2)
     we know that $\phi(Q) \neq R \cup Z$ if $Q \neq R \cup Z$. Hence $\phi$ is
     an upward closure operator on $\srl(R \cup Z)$. In particular, 
     by \cite[Cor. 10.12]{Bjoerner} $\Delta(\srl(R \cup Z))$ and 
     $\Delta(\phi(\srl(R \cup Z)))$ are homotopy equivalent.
     A subrack $Q$ of $R \cup Z$ lies in the
     image of $\phi$ if and only if $Q$ is the union of some $G$-conjugacy classes.
     Numbering the conjugacy classes $1$ up to $c$ then
     provides an identification of $\phi(\srl(R \cup Z))$ and 
     $2^{[c-z(G)+|Z|]}$ ordered
     by inclusion. Since $\Delta(2^{[c-z(G)+|Z|]})$ is a triangulation of a 
     $(c-z(G)+|Z|-2)$-sphere we are done.
  \end{proof}

  When working out the homology of $\Delta\srl(R)$ in
  \ref{ex:transposition}, \ref{ex:fourcycle} and \ref{ex:chains}, one 
  sees that this homology of is concentrated in one dimension in all three cases. 
  \ref{pr:grouprack} 
  provides another instance of the phenomenon. We next show that this is not 
  the case in general. 
  
  For $1 \leq k \leq n$, 
  $\Pi_{n,k}$ will denote the set of all
  partitions $B_1|\ldots|B_r$ in $\Pi_n$ such that, for all
  $1 \leq i \leq r$, either $|B_i| = 1$ or $|B_i| \geq k$.  
  The lattice $\Pi_{n,k}$ is called the $k$-equal partition lattice 
  (see e.g. \cite{BjoernerWelker}).

  \begin{Proposition} 
    \label{pro:pcycles}
     Let $p < n-2$ be an odd prime and $R_{n,p}$ be the rack of all 
     $p$-cycles in the alternating group $A_n$. Then the order complexes of
     $\srl(R_{n,p})$ and $\Pi_{n,p}$ are homotopy equivalent. 
   \end{Proposition}
   \begin{proof}
     Let $Q$ be a subrack of $R_{n,p}$. Set $H = \langle Q \rangle$. 
     It is well known that for odd $p$ the $p$-cycles in $A_n$ form a 
     conjugacy class.  Thus if $H = A_n$ then by \ref{le:groupcase} 
     we have $Q = R_{n,p}$. 

     \noindent {\sf Claim:} If $H = \langle Q \rangle < A_n$ then 
       $H$ does not act transitively on $[n]$. 
     
     \noindent $\triangleleft$
       Assume $H$ acts transitively on $[n]$. First, consider the case 
       that $H$ acts imprimitively. Let $B_1|
       \cdots |B_r$ be a set partition of $[n]$ stabilized by 
       $H$ where $r > 1$ and $|B_1| > 1$. If $\pi$ is a $p$-cycle in 
       $Q \subseteq H$ then the cyclic group $\langle \pi \rangle$ of order $p$
       generated by $\pi$ either stabilizes all sets 
       $B_i$ or there is an orbit $B_{i_1},\ldots,B_{i_p}$ of size $p$ on 
       the blocks. But the latter contradicts the fact that $\pi$ fixes all
       but $p$ elements in $[n]$. Thus all generators of $H$ fix the blocks of
       imprimitivity. But then $H$ cannot be transitive.
       Hence $H$ acts primitively on $[n]$. But by $p < n-2$ it follows from a 
       result by Jordan and Marggraff (see \cite[Thm. 13.9]{Wielandt}) 
       that $H$ cannot contain a $p$-cycle
       which contradicts the fact that $H$ is generated by $p$-cycles. 
       Hence $H$ cannot act transitively on $[n]$. $\triangleright$ 
  
       Consider the 
       map $\phi :  \srl(R) \rightarrow \Pi_n$ that sends $Q$ to the 
       partition of $[n]$
       given by the orbits of $H = \langle Q \rangle$. By the above reasoning 
       we know that $\phi$ restricts to a map between the proper parts of 
       $\srl(R)$ and $\Pi_n$. 
 
       \noindent {\sf Claim:} The image of $\phi$ is $\Pi_{n,p}$.

       \noindent $\triangleleft$
       Let $B_1|\cdots |B_r = \phi(Q)$ for some $Q \in \srl(R) \setminus 
       \{ \emptyset, R\}$. Let $i$ be such that
       $1 \leq |B_i| < p$. Then the elements of $B_i$ must be fixed by any
       $p$-cycle in $H$. But then $H$ must fix $B_i$ and  hence $|B_i| = 1$. 
       Thus $\phi(\srl(R)) \subseteq \Pi_{n,p}$. 
       Conversely, let $B_1| \cdots |B_r \in \Pi_{n,p}$. 
       For any $B \subseteq [n]$ denote by $A_B$ the
       alternating group on the elements of $B$. Set
       $H = A_{B_1} \times \cdots \times A_{B_r}$ and let 
       $Q_H$ be the set of all $p$-cycles contained in $H$. Then $Q_H$
       is the union of $p$-cycles with support in some $B_i$. Since the
       $p$-cycles with support in different $B_i$ commute it follows
       that $Q_H$ is a subrack with $\phi(Q_H) = B_1 |\cdots |B_r$.
       Thus the image of $\phi$ is $\Pi_{n,p}$. $\triangleright$

       Now for every 
       $\tau = B_1 |\cdots |B_r \in \Pi_{n,p}$ the lower fiber
       $\phi((\Pi_{n,p})_{\leq \tau})$ has a unique maximal element 
       $Q_H$ the set of $p$-cycles contained in 
       $H = A_{B_1} \times \cdots \times A_{B_r}$.
       Thus each fiber is a cone and hence contractible.
       Hence, by the Quillen fiber lemma \cite{Q} $\phi$ induces a homotopy
       equivalence between $\Delta(\srl(R))$ and $\Delta(\Pi_{n,p})$.
  \end{proof}
  
  \begin{Corollary}
     If $p$ is an odd prime and $p < n-2$, then 
     $\Delta(\srl(R_{n,p}))$ is homotopy equivalent to a
     wedge of spheres. If $2p \leq n$ then there are spheres
     of several dimensions in the wedge.
  \end{Corollary}
  \begin{proof}
    By \ref{pro:pcycles} $\Delta(\srl(R))$ is homotopy equivalent to 
    $\Delta(\Pi_{n,p})$. By \cite[Thm. 1.5]{BjoernerWelker}
    the assertion then follows.
  \end{proof}

\section{Classifying group classes through their subrack lattices}
  \label{sec3}

  In this section we provide the proof of \ref{main}.
  Throughout this section we will frequently use for a rack $R$ and
  subracks $S$, $T$ of $R$ the notation 
  $[S,T]_\rk$ to denote an interval $\{ Q~:~S \leq Q \leq R\}$ in 
  $\srl(R)$. 

  \subsection{Auxiliary results}
    We begin with a simple observation. The atoms in $\srl(G)$ are the 
    one-element subsets of $G$. Thus, for any $R \in \srl(G)$, $|R|$ is
    the number of atoms in the interval $[\emptyset,R]_\rk$. 
    In particular, we have the following result.

    \begin{Lemma} \label{size}
      If $\psi:\srl(G) \rightarrow \srl(H)$ is an isomorphism
      of lattices, then $|\psi(R)|=|R|$ for all $R \in \srl(G)$.
    \end{Lemma}

    For $R \in \srl(G)$, we define the {\it closure} $\ov{R}$ in 
    $\srl(G)$ to be the union of all conjugacy classes $C$ of $G$ 
    satisfying $C \cap R \neq \emptyset$. If there is any ambiguity 
    about the ambient group $G$, we will write $\ov{R_G}$ for $\ov{R}$.  
    The map on $\srl(G)$ sending $R$ to $\ov{R}$ is idempotent, increasing 
    and order preserving and thus is a closure operation as defined by Rota 
    (see \cite[p. 1852]{Bjoerner} and the proof of \ref{pr:grouprack}).
    We will say that $R$ is {\it closed} if $R=\ov{R}$.  
    So, $R$ is closed if and only if $R$ is a union of $G$-conjugacy classes.

    For any lattice $L$, we define $\il(L)$ to be the subposet of $L$ 
    consisting of those elements that are the meet of a set of coatoms of $L$.

    The following is an immediate consequence of \ref{le:maxgrrack}.

    \begin{Lemma} \label{coatom}
      Let $\Cc_1,\ldots,\Cc_k$ be the conjugacy classes in $G$.  
      \begin{enumerate}
        \item A subrack $R \in \srl(G)$ lies in $\il(\srl(G))$ if and only 
           if $R$ is closed.
       \item The lattice $\il(\srl(G))$ is isomorphic with the Boolean 
           algebra $B_k$.
    \end{enumerate}
  \end{Lemma}

  %\begin{proof}
  %  We note first that (2) and (3) follow quickly from (1). Now let 
   % $R \in \srl(G) \setminus \{G\}$. If $\la R \ra_\gp=G$ then $R$ is 
   % closed and thus lies below some $G \setminus \Cc_i$. Otherwise, 
   % $\la R \ra_\gp$ has empty intersection with some $\Cc_i$ by Burnside's 
   % Lemma, in which case 
   % $R \subseteq \la R \ra_\gp \subseteq G \setminus \Cc_i$. This in
   % turn implies $R \subseteq \la R \ra_\gp \subseteq G \setminus \Cc_i$.   
  %\end{proof}

%\subsection{Identifying non-normal maximal subgroups}

  Our next goal is to identify 
  non-normal maximal 
  subgroups of solvable groups $G$ in terms of
  the combinatorial structure of $\srl(G)$.

  We define $\nnm(G)$ to be the set of all $R \in \srl(G)$ satisfying all 
  of 
  \begin{itemize} 
    \item[(A)] $\ov{R} \neq R$, 
    \item[(B)] $[R,G]_\rk=\{R\} \bigcup [\ov{R},G]_\rk$, 
    \item[(C)] every element of $[\ov{R},G]_\rk$ is closed,
%$ [\ov{R},G]$ is a Boolean algebra, 
    \item[(D)] $\il([\emptyset,\ov{R}]_\rk)$ is not a Boolean algebra. 
  \end{itemize}
  Note that we can identify $\nnm(G)$ using only the combinatorial structure 
  of $\srl(G)$.  In particular, we have the following result.

  \begin{Lemma} \label{ord}
    Let $\psi:\srl(G) \rightarrow \srl(H)$ be an isomorphism.  
    Then $\psi(\nnm(G))=\nnm(H)$.
  \end{Lemma}

  \begin{Lemma} \label{pnnm1}
    The following claims hold. 
    \begin{enumerate} 
      \item If $R \in \nnm(G)$ then $R$ is a non-normal subgroup of $G$. 
      \item If $M$ is a non-normal maximal subgroup of $G$, then $M \in \nnm(G)$. 
%\item If $G$ is solvable and $M \in \nnm(G)$, then $M$ is a maximal subgroup of $G$.
    \end{enumerate}
  \end{Lemma}

  \begin{proof}
    Assume for contradiction that $R \in \nnm(G)$ is not a subgroup of $G$.  
    Let $N=\la R \ra$. Then every element of $[R,N]_\rk$ is closed in 
    $\srl(N)$. It follows from \ref{coatom}, applied to $N$,  
    that $[R,N]_\rk$ is a 
    Boolean algebra. As the only element of $[R,G]_\rk$ covering $R$ is 
    $\ov{R}$, we see that $\ov{R}=N$. Applying \ref{coatom} to $N$ again, 
    we see 
    that $\il([\emptyset,\ov{R}]_\rk)$ is a Boolean algebra. Thus $R$ 
    does not satisfy (D), contrary to our assumption. Therefore, $R \leq G$.  
    Also, $R \neq \ov{R}$, as $R$ is not closed. We have proved (1).

    Say $M$ is a non-normal maximal subgroup of $G$. Then $M$ is strictly 
    contained in the union of all conjugacy classes of $G$ that intersect 
    $M$ non-trivially and hence satisfies (A). If $g \in G \setminus M$, then 
    $\la M,g \ra=G$. It follows that every subrack of $G$ strictly 
    containing $M$ is closed and thus also contains $\ov{M}$. Thus $M$ 
    satisfies (B) and (C).  

%Say $M$ is a non-normal maximal subgroup of $G$.  Then $M$ satisfies (A), as $M$ is not normal in $G$.  If $g \in G \setminus M$, then $\la M,g \ra_\gp=G$.  It follows that every subrack of $G$ strictly containing $M$ is closed and thus also contains $\ov{M}$.  Thus $M$ satisfies (B) and (C).  
%Moreover, $\la \ov{M} \ra_\gp=G$.  Thus every element of $[\ov{M},G]_\rk$ is closed.  It follows from \ref{coatom} that $M$ satisfies (C).

   Let $C$ be the core of $M$ in $G$.  As every $G$-conjugate of $M$ is 
   covered by $\ov{M}$ in $\srl(G)$, we see that 
   $C \in \il([\emptyset,\ov{M}]_\rk)$. As $C$ is a closed proper subrack of $G$, 
   there is some $G$-conjugacy class $X$ such that 
   $C \subseteq \ov{M} \setminus X$.  

   Assume first that $C \neq \overline{M} \setminus X$.  As both $C$ and
   $\overline{M}$ are closed, there is some $G$-conjugacy class $X^\prime
   \neq X$ such that $X^\prime \subseteq \overline{M}$ and 
   $X^\prime \cap C=\emptyset$. 
   No $G$-conjugate of $M$ can contain $X^\prime$, as otherwise $C$ would
   contain $X^\prime$.  Therefore, no coatom of $[C,\overline{M}]_{\rk}$ containing $M
   \setminus X$ is a $G$-conjugate of $M$.  
 %  On the other hand, every such
 %  coatom contains $C$.  
   Since $C$ is the intersection of the $G$-conjugates of $M$ it follows 
   that $C$ is the meet of a proper subset
   of the set of coatoms of $[C,\overline{M}]_{\rk}$.  Therefore,
   $\il([\emptyset,\overline{M}]_{\rk})$ is not a Boolean algebra.

   %If $C \neq \ov{M} \setminus X$, then there must be a conjugacy class $X'  \neq X$ for 
   %which $X' \cap M \neq \emptyset$ and $X' \not\subseteq M$. It follows that 
   %$X' \subseteq \ov{M} \setminus X \not\subseteq M$. In this case, there is some coatom 
   %of $[C,\ov{M}]_\rk$ that is not a $G$-conjugate of $M$. It follows that 
   %$\il([C,\ov{M}]_\rk)$ is not a Boolean algebra, and in turn that 
   %$\il([\emptyset,\ov{M}]_\rk)$ is not a Boolean algebra.

%As every $G$-conjugate of $M$ is covered by $\ov{M}$ in $\srl(G)$, we see that $C \in \il([\emptyset,\ov{M}]_\rk)$.  As $C$ is a closed subrack of $G$, there is some $G$-conjugacy class $X$ such that $C \subseteq \ov{M} \setminus X$.  If $C \neq \ov{M} \setminus X$, then $\ov{M} \setminus X \not\subseteq M$.  (Indeed, if a closed subrack of $G$ is contained in $M$, then it is contained in every $G$-conjugate of $M$.)   In this case, there is some coatom of  $[C,\ov{M}]_\rk$ that is not a $G$-conjugate of $M$.  It follows that $\il([C,\ov{M}]_\rk)$ is not a Boolean algebra, and in turn that $\il([\emptyset,\ov{M}]_\rk)$ is not a Boolean algebra.

    Assume now that $C=\ov{M} \setminus X$. Then every non-identity element 
    of $M/C$ is of the form $Cx$ with $x \in X$. It follows that all such 
    elements are conjugate in $G/C$. Thus all such elements have the same 
    order, which must be prime. Therefore, there is a prime $p$ such that 
    $M/C$ is a $p$-group of exponent $p$. We claim that $p=2$. Indeed, 
    assume for contradiction that $p>2$ and let $Cz$ be a non-identity element  
    of $Z(M/C)$. Then $Cz$ is conjugate to $Cz^{-1}$ in $G/C$, but not in 
    $M/C$. It follows that $\langle Cz \rangle$ is normal in $G/C$, as it is 
    normalized by both the maximal subgroup $M/C$ and by some element not in 
    $M/C$. This contradicts the fact that $M/C$ has trivial core in $G/C$.

    We see now that $M/C$ is an elementary abelian $2$-group. As $M$ is 
    maximal and non-normal in $G$, it follows that $M/C$ is Sylow 
    $2$-subgroup of $G/C$ and $N_{G/C}(M/C)=M/C$. By Burnside's Normal 
    $p$-Complement Theorem (see for example \cite[Theorem 7.2.1]{KuSt}), 
    $G/C$ contains a normal complement $N/C$ to $M/C$. As $M/C$ is maximal 
    in $G/C$, we see that $N/C$ is characteristically simple and thus the 
    direct product of pairwise isomorphic simple groups (see for example 
    \cite[1.7.3]{KuSt}). Now by either the Feit-Thompson odd order Theorem 
    or \cite[Lemma 3.24]{BaLu}, $N/C$ is an elementary abelian $p$-group for 
    some odd prime $p$.  

    The conjugation action of $M/C$ on $N/C$ determines a linear 
    representation $\phi$ of $M/C$ on the $\FF_p$-vector space $N/C$.  
    As $N/C$ is minimal normal in $G/C$ and $M/C$ has trivial core in 
    $G/C$, $\phi$ is faithful and irreducible.  It follows $|M/C|=2$ (see for 
    example \cite[0.5]{MaWo}) and in turn that $|N/C|=p$. As $M$ is not normal 
    in $G$, we see that $G/C$ is dihedral of order $2p$. Thus $M$ has 
    $p \geq 3$ conjugates in $G$, and the intersection of any two of these 
    conjugates is $C$. It follows that $\il([C,\ov{M}]_\rk)$ is not a 
    Boolean algebra. Thus $M$ satisfies (D) and (2) holds.

  \end{proof}

  \begin{Corollary} \label{maxcor}
    If $M$ is a maximal element of $\nnm(G)$ with respect to the order 
    inherited from $\srl(G)$, then $M$ is a non-normal subgroup of $G$ 
    and every subgroup of $G$ properly containing $M$ is normal in $G$.
  \end{Corollary}

  \begin{proof}
    This follows directly from \ref{pnnm1}.
  \end{proof}

  We will prove something stronger than \ref{maxcor} under the assumption 
  that $G$ is solvable.

  \begin{Lemma} \label{diffclo}
    Let $L,M$ be non-conjugate maximal subgroups of the solvable group $G$.  
    Then $\ov{L} \neq \ov{M}$.
  \end{Lemma}

  \begin{proof}
    By a Theorem of Ore (see \cite[(16.1)]{DoHa} or \cite{O}), 
    $L$ and $M$ have different cores in $G$. Let $C=\Core_G(M)$. 
    We may assume that $C \not\leq L$. Let 
    \[ 1=N_0 \lhd N_1 \lhd \ldots \lhd N_\ell=G \] be a chief series for 
    $G$ such that $N_j=C$ for some $j \in [\ell]$. Find the smallest 
    $i \in [\ell]$ such that $N_i \not\leq L$. Note $i \leq j$. Then 
    $L/N_{i-1}$ is a complement to $N_i/N_{i-1}$ in $G/N_{i-1}$ and 
    thus $L \cap N_i=N_{i-1}$. It follows that $L$ has empty intersection 
    with every $G$-conjugacy class in $N_i \setminus N_{i-1}$.  
    However, $N_i \leq C \leq M$.
  \end{proof}

  \begin{Lemma} \label{selfnor}
    If $M$ is a maximal element of $\nnm(G)$ with respect to the order 
    inherited from $\srl(G)$, then $N_G(M)=M$.
  \end{Lemma}

  \begin{proof}
    Let $K$ be the intersection of all subgroups of $G$ properly containing 
    $M$. Then $K \unlhd G$ by \ref{maxcor}. So, $K \neq M$ and $M$ is a 
    maximal subgroup of $K$. Moreover, if $N_G(M) \neq M$, then 
    $K \leq N_G(M)$.
 
    Assume for contradiction that $M \lhd K$. Thus $M$ is a union of 
    $K$-conjugacy classes. If $K \setminus M$ contains more than one 
    $K$-class, then $M$ is covered by more than one element in $[M,K]_\rk$.  
    This is impossible, as $M$ is covered only by $\ov{M}$ in $[M,G]_\rk$.  
    If $K \setminus M$ consists of one $K$-class, then $M$ is a maximal 
    subrack of $K$ and $K=\ov{M}$. This is also impossible, 
    as $\il([\emptyset,K]_\rk)$ is a Boolean algebra by \ref{coatom},         
    while $M$ satisfies (D).
  \end{proof}

  \begin{Corollary} \label{solvmax}
    If $G$ is solvable and $M \in \nnm(G)$, then $M$ is a non-normal maximal 
    subgroup of $G$.
  \end{Corollary}

  \begin{proof}
    Again, let $K$ be the intersection of all subgroups of $G$ properly 
    containing $M$. Assume for contradiction that $K \neq G$. By 
    \ref{selfnor}, $M$ has $[K:M]$ $K$-conjugates and $[G:M]$ $G$-conjugates.  
    Thus there exists a $G$-conjugate $L$ of $M$ that is not a $K$-conjugate 
    of $M$. As $K \lhd G$, we know that $L \leq K$. Moreover, 
    $\ov{L_G}=\ov{M_G}$ is the unique element covering $L$ in $\srl(G)$. 
    It follows that
    \[
      \ov{L_K}=\ov{L_G}=\ov{M_G}=\ov{M_K}.
    \]
    This is impossible by \ref{diffclo}, as $K \leq G$ is solvable.
 \end{proof}

\subsection{Proof of \ref{main}}

  \begin{proof}[Proof of \ref{main}(1)] 
    If $G$ is abelian then every subset of $G$ is a subrack of $G$ and thus 
    $\srl(G)$ is a Boolean algebra. On the other hand, say the elements 
    $x,y \in G$ do not commute. Then $\{x,y\}$ is not a subrack of $G$.  
    As the atoms in $\srl(G)$ are exactly the one-element subsets of $G$, 
    it follows that $\srl(G)$ is not a Boolean algebra.  
    \ref{main}(1) follows.
  \end{proof}

  To prove that \ref{main}(2) holds, we use the following result 
  (see \cite[Exercise 4, p. 107]{KuSt}).

  \begin{Theorem} \label{hup}
    A finite group $G$ is nilpotent if and only if every maximal subgroup 
    of $G$ is normal in $G$.
  \end{Theorem}

  \begin{proof}[Proof of \ref{main}(2)]
    As a nilpotent group is solvable, it follows from \ref{hup}, 
    \ref{pnnm1}(2) and \ref{solvmax} that $G$ is nilpotent if and only 
    if $\nnm(G)=\emptyset$.
  \end{proof}

  For the proof of the solvable case \ref{main}(4) we will need the following 
  result of Kano (see \cite[Theorem 1]{Ka}).

  \begin{Theorem} \label{kano}
    If every non-normal maximal subgroup of $G$ has the same order, 
    then $G$ is solvable.
  \end{Theorem}

  \begin{proof}[Proof of \ref{main}(4)]
    We prove \ref{main}(4) by induction on $|G|$, the base case $G=1$ being 
    trivial. 

    Assume that $G$ is solvable and let $\psi:\srl(G) \rightarrow \srl(H)$ 
    be an isomorphism. By \ref{main}(2), we may assume that $G$ is 
    not nilpotent. Now by \ref{hup}, $G$ has at least one conjugacy 
    class of non-normal maximal subgroups.

    Assume first that $G$ has two non-conjugate, non-normal maximal subgroups, 
    $M_1$ and $M_2$. For $i=1,2$, set $N_i=\Core_G(M_i)$, $L_i=\psi(M_i)$ and 
    $D_i=\Core_H(L_i)$. 
   
%    We have the following simple facts. 
%    \begin{itemize}
%      \item Each $L_i$ is a subgroup of $G$ by \ref{maxcor}.
%      \item By \ref{coatom} each $N_i$ is the intersection of all 
%        coatoms containing $M_i$. Thus $\psi(N_i) = D_i$ for $i =1,2$.
%    \end{itemize}
%
%{\color{blue} I DON'T UNDERSTAND THE SECOND BULET POINT ABOVE.  SAY $G=S_3$ AND $M=S_2$.  THEN $M$ HAS TRIVIAL CORE IN $G$, BUT THE ONLY COATOM ABOVE $M$ IS $\overline{M}$, RIGHT? HOW ABOUT THE FOLLOWING:

    We have the following facts. 
    \begin{itemize}
      \item Each $L_i$ is a subgroup of $H$ by \ref{maxcor}.
      \item For $i=1,2$, $\psi(\overline{M_i})=\overline{L_i}$.  Indeed,
      $[L_i,H]_\rk \cong [M_i,G]_\rk$.  As $\overline{M_i} \in \il(\srl(G))$,
      it follows that $\psi(\overline{M_i}) \in \il(\srl(H))$ is closed.  Moreover,
      $\psi(\overline{M_i})$ covers $L_i$.
      \item For $i=1,2$, $\psi(N_i)=D_i$.  To see this, note that by \ref{diffclo},
      the intersection of the set of coatoms in $[\emptyset,\overline{M_i}]$ with $\nnm(G)$
      is the set of $G$-conjugates of $M_i$.  The map $\psi$ sends this intersection to the
      intersection of the set of coatoms in $[\emptyset,\overline{L_i}]$ with $\nnm(H)$.
      As $|G|=|H|$ and $|M_i|=|L_i|$, it follows from \ref{selfnor}, that this second intersection
      consists exactly of the $H$-conjugates of $L_i$.  Now taking the meets in the respective
      lattices of the two intersections in question yields the claim.
      %\item By \ref{coatom} each $C_i$ is the intersection of all 
        %coatoms containing $M_i$. Thus $\psi(C_i) = D_i$ for $i =1,2$.
    \end{itemize}

    Now $N_1 \neq N_2$, as $M_1,M_2$ are not conjugate.  
    We may assume that $N_1 \not\leq M_2$.  Therefore, $D_1 \not\leq L_2$.  

    By \ref{diffclo}, the only elements of $\nnm(H)$ with the same closure as 
    $L_1$ are those of the form $\psi(X)$ with $X$ conjugate to $M_1$ in $G$.  
    By \ref{selfnor}, we get $N_H(L_1)=L_1$. As $|G|=|H|$ and 
    $|L_1|= |M_1|$ by \ref{ord}, we see that $D_1$ is the intersection of the 
    $H$-conjugates of $L_1$.  Thus $D_1 \lhd H$. Now $L_2<D_1L_2 \leq H$. 
    As $L_2$ is maximal in $\nnm(H)$, we see that $D_1L_2 \unlhd H$.  
    Moreover, $H/D_1L_2$ is nilpotent by \ref{hup}.  
    By inductive hypothesis, $D_1$ is solvable, as is 
    $D_1L_2 /D_1\cong L_2/(D_1 \cap L_2)$. Thus $H$ is solvable.

    Now assume that $G$ has exactly one conjugacy class of non-normal maximal 
    subgroups. Then all maximal elements of $\nnm(H)$ (with respect to the 
    order inherited from $\srl(H)$) have the same order, by \ref{ord} and 
    \ref{maxcor}. Therefore, $H$ is solvable by \ref{kano}. 
  \end{proof}

  \begin{proof}[Proof of \ref{main}(3)]
    Say $G$ is supersolvable and $\psi:\srl(H) \rightarrow \srl(G)$ is an isomorphism. 
    Then $H$ is solvable 
    by \ref{main}(4). Thus every $X \in \nnm(H)$ is maximal and non-normal in 
    $H$. It follows that $\psi^{-1}(X)$ is maximal and non-normal in $G$.  
    Thus $[G:\psi^{-1}(X)]$ is prime, as $G$ is supersolvable. Now $[H:X]$ is 
    also prime, as $|H|=|G|$ and $|X|=|\psi^{-1}(X)|$. We see now that every 
    maximal subgroup of $H$ has prime index. It follows from a result of 
    Huppert (see \cite[Satz 9]{Hu}) that $H$ is supersolvable.  
  \end{proof}

  It remains to prove \ref{main}(5). 

  \begin{proof}[Proof of \ref{main}(5)]
    Assume that $G$ is simple and $\srl(H) \cong \srl(G)$. If $G$ is abelian 
    then $G$ has prime order and so does $H$. Assume from now on that $G$ is 
    non-abelian simple. For any group $X$ and any $x \in X$, we will write 
    $\cl_X(x)$ for the conjugacy class of $x$ in $X$. We observe that every 
    non-trivial conjugacy class of $G$ generates $G$ as a group. It follows 
    that if $C \subseteq G$ is a non-trivial conjugacy class and $g \in G$, 
    then the join of $C$ and $\{g\}$ in $\srl(G)$ is $C \bigcup \cl_G(g)$.  
    It follows from \ref{coatom} that if $C \in \srl(H)$ is an atom in 
    $\il(\srl(H))$ and $h \in H$, then the join of $C$ and $\{h\}$ in 
    $\srl(H)$ is the union of $C$ and the unique atom of $\il(\srl(H))$ 
    lying above $\{h\}$ in $\srl(H)$.

    Let $N$ be a minimal normal subgroup of $H$. Thus 
    $$N=N_1 \times \ldots \times N_t,$$ with the $N_i$ being pairwise 
    isomorphic simple groups.

    We claim first that $N$ is not central in $H$. In fact, $Z(H)=1$.  
    Indeed, $Z(H)$ consists of those $h \in H$ such that 
    $\{h,x\} \in \srl(H)$ for all $x \in H$. As $Z(G)=1$, $\srl(G)$ has one 
    such element, and the claim follows.

    Next, assume for contradiction that $N$ is abelian, that is, each $N_i$ 
    has prime order $p$. By \ref{main}(2), we know that $H$ is not a 
    $p$-group. There is some $h \in H$ of prime order $r \neq p$. Pick some 
    $x \in N \setminus 1$. Note that $\cl_H(x) \subseteq N$. Let 
    $K=\langle h,\cl_H(x) \rangle$. The join of $\{h\}$ and $\cl_H(x)$ in 
    $\srl(H)$ is the union of $\cl_H(x)$ and $\cl_K(h)$. Therefore, 
    $\cl_K(h)=\cl_H(h)$.  Now, as $K \leq \langle h \rangle N$, 
    $[K:C_K(h)]$ is a $p$-group.  Therefore, $|\cl_H(h)|$ is a power of $p$.  
    However, the multiset of conjugacy class sizes in $H$ is the same as 
    that of $G$. As $G$ is simple, no non-trivial conjugacy class in $G$ has 
    prime power order, as shown by Burnside in \cite[p. 392]{Bu}.

    We assume now that each $N_i$ is non-abelian simple and assume further 
    for contradiction that $t>1$. Pick $x \in N_1$ and some non-trivial 
    conjugacy class $C$ of $H$ that is contained in $N$ but does not contain 
    $x$. (Such $C,x$ exist, since $N$ is not elementary abelian.)  
    Let $L=\langle C,x \rangle$. Then the join of $C$ and $x$ in 
    $\srl(G)$ is $C \cup \cl_L(x)$. This forces $\cl_L(x)=\cl_H(x)$.  
    However, $\cl_L(x) \subseteq N_1$, while $\cl_H(x) \not\subseteq N_1$, 
    as $H$ acts transitively on $\{N_1,\ldots,N_t\}$ by the minimality of $N$.

    We conclude that every minimal normal subgroup of $H$ is non-abelian 
    simple. Assume for contradiction that $M,N$ are distinct minimal normal 
    subgroups of $H$. Then $M$ and $N$ commute elementwise. Pick a non-trivial 
    conjugacy class $C$ of $G$ contained in $M$ and a non-trivial element 
    $x \in N$. Let $J=\langle C,x \rangle$. The join of $C$ and $\{x\}$ 
    in $\srl(H)$ is $C \bigcup \cl_J(x)$.  It follows that 
    $\cl_H(x)=\cl_J(x)$.  However, $x \in Z(J)$ while $x \not\in Z(H)$.

    We see now that $H$ has a unique minimal normal subgroup $N$ and $N$ is 
    non-abelian simple, that is, $H$ is almost simple.  Assume for 
    contradiction that $H \neq N$. A result of Feit and Seitz  
    (\cite[Theorem C]{FeSe}) says that there is some conjugacy class $C$ of 
    $N$ such that $C$ is not a conjugacy class in $H$. Pick some $x \in C$ 
    and some non-trivial conjugacy class $D \neq C$ of $N$. As above, the 
    join of $D$ and $x$ in $\srl(H)$ is contained in $D \bigcup C$.  
    However, we know already that this join is $D \bigcup \cl_H(x)$.  
    This completes the proof of \ref{main}(5).
  \end{proof}

%Before completing the proof, we note that we have not invoked the Classification.  Moreover, we can conclude now that $H/N$ is nilpotent.  Indeed, otherwise there is some non-normal maximal subgroup $M$ of $H$ such that $N \leq M$.  We can identify (some conjugate of) $M$ in $\srl(H)$ using \ref{pnnm1}.  Moreover, if 

\subsection{Final comments on the proof or \ref{main}}

  Let us present a second proof of \ref{main}(2). 
  The conjugacy classes in $G$ are exactly the atoms in $\il(\srl(G))$.  
  Thus we can determine the multiset of conjugacy class sizes in $G$ from 
  $\srl(G)$. According to a result of Cossey, Hawkes and Mann in \cite{CHM}, 
  this multiset tells us whether or not $G$ is nilpotent.  As far as we know, 
  it is not known if the multiset of class sizes tells us whether or not $G$ 
  is supersolvable, or if it tells us whether or not $G$ is solvable.

  For ease of exposition, we invoke the Feit-Thompson Odd Order 
  Theorem at one point. As we indicate at that point, 
  we can obtain the same consequence by more elementary means.

  To finish our proof of \ref{main}(5), we invoked \cite[Theorem C]{FeSe}, 
  the proof of which uses the Classification. Without the Classification, we 
  have shown that if $G$ is simple and $\srl(H) \cong \srl(G)$, then $H$ is 
  almost simple. We can conclude in addition that $H/F^\ast(H)$ is nilpotent.  
  Indeed, otherwise some non-normal maximal subgroup $M<H$ contains 
  $F^\ast(H)$. It follows that there exists some 
  ${\mathcal S} \subseteq \nnm(H)$  such that $\bigcap_{S \in {\mathcal S}}S$ 
  is closed in $\srl(H)$ and has more than one element. Indeed, the 
  intersection of all $H$-conjugates of $M$ is $F^\ast(H)$. As 
  $\srl(G) \cong \srl(H)$, the intersection of some elements of $\nnm(G)$ is 
  a closed subrack of $G$ containing more than one element. As $G$ is simple, 
  this is impossible by \ref{pnnm1}(1).

\section{Gradedness of subrack lattice}
  \label{sec4}

  In this section we will prove \ref{graded} after stating and proving a 
  series of preliminary results.  

  If $G$ is abelian, then $\srl(G)=2^G$ is graded.  It is reasonable to hope 
  that all non-abelian groups have non-graded subrack lattices, and one might 
  try  to prove that this is the case by 
  showing that minimal non-abelian groups have non-graded subrack lattices.  
  An examination of this class finds three exceptions, namely, the three 
  smallest non-abelian groups, $S_3$, $D_8$ and $Q_8$.  
  Indeed, the next result can easily be proved by direct inspection.  

  \begin{Proposition} \label{s3d8q8}
    If $G \cong S_3$, then $\srl(G)$ is graded and each maximal 
    chain has length $3$. 
    If $G$ is isomorphic to one of 
    $D_8$ or $Q_8$, then $\srl(G)$ is graded and each maximal chain has
    length $5$. 
  \end{Proposition}

  The main content of \ref{graded} is that these three are the only 
  non-abelian groups with graded subrack lattices.  
  For its proof we will use the following lemma repeatedly.

  \begin{Lemma} \label{maxsg}
    Let $M$ be a maximal subgroup of the finite group $G$.  Let $c$ be the number of conjugacy classes $C$ 
    of $G$ such that $M \cap C=\emptyset$.  
    \begin{enumerate} 
      \item If $M \lhd G$, then every maximal chain in $[M,G]_\rk$ has length $c$. 
      \item If $M$ is not normal in $G$, then every maximal chain in $[M,G]_\rk$ has length $c+1$. 
    \end{enumerate}
  \end{Lemma}

  \begin{proof}
    As $M$ is maximal in $G$, $\langle g,M \rangle=G$ for all $g \in G \setminus M$.  
    Therefore, every element of $(M,G]_\rk$ is a union of $G$-conjugacy classes.  
    Claim (1) follows immediately, and (2) follows from \ref{pnnm1}(2).
  \end{proof}

  In \cite{MiMo} the minimal non-abelian finite groups are described. 
  For our purposes, the following consequence will suffice.

  \begin{Proposition} \label{mm}
    If $G$ is a minimal non-abelian finite group, then either $G$ has prime power order, 
    or $G=CA$, where $A$ is a normal elementary abelian subgroup, $C$ is cyclic and both $A$ and $C$ have prime power order.
  \end{Proposition}

  \begin{Lemma} \label{minp}
    Let $p$ be a prime and let $G$ be a minimal non-abelian $p$-group.  The lattice $\srl(G)$ is graded if and only if $|G|=8$.
  \end{Lemma}

  \begin{proof}
    Assume that $G$ is minimal non-abelian of order $p^k$ and that $\srl(G)$ is graded. Note that $k \geq 3$.

    As $G$ is not cyclic, $G$ has more than one maximal subgroup.  Let $A,B$ be maximal subgroups of $G$.  
    Then both $A$ and $B$ are abelian normal subgroups of index $p$ in $G$.  As $G=AB$, it follows that $Z:=A \cap B$ is 
    contained in the center $Z(G)$ of $G$. As $G$ is non-abelian, $[G:Z(G)] \geq p^2$. From $[G:Z]=p^2$ it follows that $Z=Z(G)$.  
    Since $A,B$ were arbitrary maximal subgroups, we see that $Z \leq M$ for every maximal subgroup $M$ of $G$.  %So, $Z$ is the Frattini subgroup $\Phi(G)$.  Thus $G/\Phi(G)$ is $2$-generated, and so is $G$.

    Observe that if $g \in G \setminus Z$ then $|\cl_G(g)|=p$.  Indeed, any maximal subgroup of $G$ containing $g$ is abelian and thus is the centralizer of $g$.  
    It follows that if $A$ is any maximal subgroup of $G$, then $$\frac{p^k-p^{k-1}}{p}=p^{k-1}-p^{k-2}$$ $G$-conjugacy classes intersect $A$ trivially.  
    As $A$ is abelian and normal in $G$, it follows from \ref{maxsg} that $\srl(G)$ has a maximal chain of length $2p^{k-1}-p^{k-2}$ 

    Now fix two elements $x,y \in G$ that do not commute.  As $G$ is minimal non-abelian, $G=\langle x,y \rangle$.  Let $M$ be a maximal subgroup of $G$ 
    containing $x$. Note that $M \lhd G$, $Z \leq M$, and $y \not\in M$.

    As $G=\langle x,y \rangle$ is non-abelian, $x \not\in Z$.  So, $|\cl_G(x)|=p$.  As $x \in M$ and $M$ is abelian, every subset of $\cl_G(x)$ is a subrack of $G$.  
    In particular, there exists a chain $\{R_i:0 \leq i \leq p\}$ of subracks of $\cl_G(x)$ such that $|R_i|=i$ for each $i$. In particular, we have $R_p=\cl_G(x)$.  

    If $Q$ is any subrack of $G$ containing $R_p$ and $y$, then $\langle Q \rangle=G$ and thus $Q$ is a union of $G$-conjugacy classes.  
    It follows that $Q_y:=R_p \cup \cl_G(y)$ covers $R_p$ in $\srl(G)$ and that any saturated chain in $[Q_y,G]_\rk$ is obtained by adding, one at a time, 
    the $G$-conjugacy classes other than $\cl_G(x)$ and $\cl_G(y)$ to $Q_y$. As $G$ has $$p^{k-2}+\frac{p^k-p^{k-2}}{p}=p^{k-1}+p^{k-2}-p^{k-3}$$ conjugacy classes, 
    it follows now that $\srl(G)$ has a chain of length $$p+p^{k-1}+p^{k-2}-p^{k-3}-1.$$  

    Since $\srl(G)$ is graded, we conclude that $$2p^{k-1}-p^{k-2}=p^{k-1}+p^{k-2}-p^{k-3}+p-1,$$ equivalently, that $$p^{k-3}(p-1)^2=p-1,$$ which forces $p=2$ and $k=3$ as claimed.
  \end{proof}

  \begin{Lemma} \label{minsdp}
    Let $p,q$ be primes and let $G=CA$ be a minimal non-abelian group such that $C \leq G$ is a cyclic $p$-subgroup and $A \lhd G$ is an elementary 
    abelian normal $q$-subgroup. The lattice $\srl(G)$ is graded if and only if $G \cong S_3$.
  \end{Lemma}

  \begin{proof}
    Let $c$ generate $C$ and consider $A$ to be a vector space over $\FF_q$. Assume that $\dim_{\FF_q} A=\ell$, so $|A|=q^\ell$, and that $|C|=p^k$.

    As $G$ is minimal non-abelian, $c^p$ must centralize $A$ and $C$ must act irreducibly on $A$. It follows that either $p \neq q$ or
    $p=q$ and $A = \FF_p$. But in the latter case $C$ centralizes $A$ and $G$ is abelian. Hence $p \neq q$ and $C_A(c)=1$.  
    Therefore, $C$ has $q^\ell$ $G$-conjugates, and each of the $q^\ell(p^k-p^{k-1})$ elements of order $p^k$ in $G$ is conjugate to exactly one generator of $C$.  
    Thus the elements of $G$ having order $p^k$ form $p^k-p^{k-1}$ conjugacy classes.

    The remaining $p^{k-1}q^\ell$ elements of $G$ lie in the normal, abelian maximal subgroup $M:=\langle c^p,A \rangle$ of $G$.  The elements of 
    $\langle c^p \rangle$ are central in $G$, and the remaining elements of $M$ form $p^{k-2}(q^\ell-1)$ conjugacy classes, each class having size $p$.

    Applying \ref{maxsg} to 
    \begin{itemize} 
      \item $M$, we see that $\srl(G)$ has a maximal chain of length $$p^{k-1}q^\ell+p^k-p^{k-1}=p^{k-1}(q^\ell+p-1).$$

      \item the non-normal maximal subgroup $C$, we see that $\srl(G)$ has a maximal chain of length $$p^k+1+p^{k-2}(q^\ell-1).$$
    \end{itemize}

    Assuming $\srl(G)$ is graded, we see that $$p^{k-1}(q^\ell+p-1)=p^k+1+p^{k-2}(q^\ell-1).$$  If $k>1$, we get a contradiction from reduction 
    modulo $p$.  Therefore, if $\srl(G)$ is graded then $k=1$ and $$q^\ell+p-1=p+1+\frac{q^\ell-1}{p},$$ from which it follows that 
    $$\frac{q^\ell}{2}<2.$$  Therefore, $q^\ell=3$ and $p^k=2$ and $G = S_3$. The Lemma now follows from \ref{s3d8q8}.
  \end{proof}

  Combining \ref{mm}, \ref{minp} and \ref{minsdp}, we obtain the following result.

  \begin{Proposition} \label{mns}
    If $G$ is a non-abelian finite group and $\srl(G)$ is graded, then every minimal non-abelian subgroup of 
    $G$ is isomorphic to one of $S_3$, $D_8$ or $Q_8$. In particular, $|G|$ is even.
  \end{Proposition}

  This immediately implies. 

  \begin{Corollary} \label{oddorder}
    Let $G$ be a group of odd order.  The lattice $\srl(G)$ is graded if and only if $G$ is abelian.
  \end{Corollary}

  %\begin{proof}
   %As $\srl(G)=2^G$ for any abelian group $G$, we need only prove that $G$ is abelian when $G$ has odd order and $\srl(G)$ is graded.  Assume for contradiction that $G$ is a counterexample to this claim of minimal order.    So, $G$ is non-abelian but $\srl(G)$ is graded.  For each proper subgroup $H<G$, $\srl(H)$ is the interval $[\emptyset,H]_\rk$ in $\srl(G)$ and is therefore graded.  By minimality of $|G|$, $H$ is abelian.  Thus $G$ is minimal non-abelian and therefore (by a 1903 result of Miller and Moreno) solvable (we could also use the odd order theorem).

  %Now let $N$ be a normal subgroup of (odd) prime index $p$ in $G$ and fix $y \in G \setminus N$.  Write $|y|=p^k\ell$ with $\gcd(p,\ell)=1$.  As $y^p \in N$, $x:=y^\ell \not\in N$ and $x$ has order $p^k$.  So, $H=\langle N,x \rangle$  As $N$ is abelian, each Sylow subgroup of $N$ is characteristic.  Thus, if $Q$ is a Sylow $q$-subgroup of $N$, $x$ normalizes $Q$.  If $\langle x \rangle Q \neq H$, then $\langle x \rangle Q$ is abelian and, in particular, $Q$ centralizes $x$.  It follows that if $N$ does not have prime power order, then $N$ centralizes $x$.  However, this contradicts the fact that $H$ is non-abelian.  So, $N$ has order $q^j$ for some prime $q$.  If $q=p$ then \ref{minp} yields the desired contradiction, while if $q \neq p$ we obtain this contradiction from \ref{minsdp}.
  %\end{proof}

  \ref{minp} allows us to determine the structure of a $2$-group whose subrack lattice is graded.

  \begin{Lemma} \label{nab2}
    If $G$ is a non-abelian $2$-group and $\srl(G)$ is graded, then $|G|=8$.
  \end{Lemma}

  \begin{proof}
    When showing that $\srl(G)$ is not graded if $|G|>8$, we may assume that $|G|=16$.  Indeed, under the assumptions 
    of the theorem, every minimal non-abelian subgroup of $G$ has order eight by \ref{minp}. Such a subgroup is contained in 
    some $H \leq G$ such that $|H|=16$, and $H$ satisfies the conditions of the theorem.

    Now, as $|G|=16$, $G$ has an abelian subgroup $A$ of order eight (see for example (4.26) in \cite{Suzuki}).  Moreover, $G$ 
    contains a minimal non-abelian subgroup $E$ of order eight.  

    We consider the conjugacy classes in $G$.  Let $Z=Z(G)$.  As $G$ is not abelian, we see that $Z \leq A$.  Let  
    $t \in G \setminus A$ then $\langle A,t\rangle = G$ and $G = A \langle t \rangle$. Since $G$ is non-abelian, we see that $Z=C_A(t)$. 
    It follows that $C_G(t)=\langle t \rangle Z$.  As $A$ has index $2$ in $G$ and $G = A\langle t \rangle$ we have $t^2 \in A$
    and hence $t^2 \in C_A(t)$. It follows that $[C_G(t):Z]=2$ and that $$|\cl_G(t)|= \frac{|G|}{2|Z|}.$$  
    From this we deduce that the elements of 
    $G \setminus A$ form $|Z|$ conjugacy classes of equal size. 

    If $a \in A \setminus Z$ then $C_G(a)=A$ and therefore $|\cl_G(a)|=2$.  Hence, the elements of $A$ form $|Z|$ conjugacy classes of size one 
    and $4-\frac{|Z|}{2}$ classes of size two.

    We now exhibit two maximal chains of different length, from which it follows that $\srl(G)$ is not graded.

    \begin{itemize}
      \item An application of \ref{maxsg} to $A$ yields a maximal chain in $\srl(G)$ of length $8+|Z|$.
      \item Next we apply \ref{maxsg} to $E$.  Note that $|E \cap A|=4$ and $|E \cap Z|=2$.  Therefore, the elements of $A \setminus E$ 
         form $|Z|-2$ central conjugacy classes and $3-\frac{|Z|}{2}$ non-central classes.  The elements of $G \setminus (A \cup E)$ 
         form $\frac{|Z|}{2}$ classes, as these are half of the elements of $G \setminus A$.  We see now that $\srl(G)$ 
         contains a chain of length $$6+(|Z|-2)+(3-\frac{|Z|}{2})+\frac{|Z|}{2}=7+|Z|.$$
    \end{itemize}

  \end{proof}

  With the local structure of a group with graded subrack lattice in hand, we turn to the global structure.

  \begin{Lemma} \label{npc}
    Let $G$ be a finite group such that $\srl(G)$ is graded.  
    If $p>3$ is a prime divisor of $|G|$, then $G$ has a normal $p$-complement.
  \end{Lemma}

  \begin{proof}
    Let $P$ be a Sylow $p$-subgroup of $G$.  By \ref{oddorder}, $P$ is abelian.  
    Let $x \in N_G(P)$. 

    \noindent {\sf Claim:} $x \in C_G(P)$. 

    $\triangleleft$ Assume for contradiction that $x \not\in C_G(P)$.  As $\langle x \rangle$ is the product of its Sylow subgroups, some
    $y \in \langle x \rangle$ has prime power order and does not centralize $P$.  Say $|y|=q^k$ with $q$ prime.  
    Note that $q \neq p$.  The group $\langle y \rangle P$ is non-abelian and thus contains a minimal non-abelian subgroup $K$.  
    As $P$ and $\langle y \rangle$ are abelian,  $K$ is neither a $p$-group nor a $q$-group.  By \ref{mm} $K$ is the semi-direct 
    product of a elementary abelian normal $p$-subgroup and a cyclic $q$-group.  As $p>3$, this contradicts \ref{minsdp}.  
    It follows $x \in C_G(P)$. $\triangleright$ 

    Hence $P$ is central in $N_G(P)$. The lemma now follows from Burnside's Normal 
    $p$-Complement Theorem (see Theorem 7.2.1 in \cite{KuSt}). 
  \end{proof} 

  \begin{Proposition} \label{solvable}
    If $G$ is a finite group and $\srl(G)$ is graded, then $G$ is solvable.
    %and a Hall $\{2,3\}$-subgroup of $G$ is normal.
  \end{Proposition}

  \begin{proof}
    We prove $G$ is solvable by induction on the number $\pi$ of distinct prime divisors of $|G|$.  If $\pi \leq 2$ then $G$ is 
    solvable by Burnside's $p^aq^b$ Theorem (see \cite[10.2.1]{KuSt}).  Now assume $\pi>2$. Then $G$ must have a prime divisor $p>3$. 
    By \ref{npc}, $G=PN$, where $P$ is a Sylow $p$-subgroup of $G$ and $N$ is a normal $p$-complement. As $\srl(N)$ is graded, 
    $N$ is solvable by inductive hypothesis. Moreover, $G/N$ is a $p$-group and therefore solvable.  Thus $G$ is solvable.  
   %Now let $\Pi_{>3}$ be the set of all prime divisors $p$ of $G$ such that $p>3$.  With $N_p$ as defined above, $\bigcap_{p \in \Pi_{>3}}N_p$ is a normal Hall $\{2,3\}$-subgroup of $G$.
  \end{proof}

   %We will need to show that certain small groups do not have graded subrack lattices.  To that end, we use the following lemma.
  Now we use \ref{maxsg} to prove that certain groups,
  which will appear later in our argument, do not have graded subrack
  lattices.

  \begin{Lemma} \label{small}
    Let $p$ be a prime.  The subrack lattice $\srl(G)$ is not graded if $G$ is isomorphic to one of 
    \begin{enumerate} 
      \item $S_3 \times \ZZ_p$, 
      \item $D_8 \times \ZZ_p$ or $Q_8 \times \ZZ_p$, 
      \item a dihedral group of order $2k$ with $k>3$ odd, 
      \item the semi-direct product $TV$, where $T= \langle t \rangle$ has order two, 
            $V \cong \ZZ_3 \oplus \ZZ_3$ and $t^{-1}vt=v^{-1}$ for all $v \in V$, or 
      \item $\SL_2(3)$. \end{enumerate} 
  \end{Lemma}

  \begin{proof}
    \begin{enumerate}
      \item 
    If $G=S_3 \times \ZZ_p$, then $G$ has $p$ central conjugacy classes, $p$ classes of size two 
    (one consisting of elements of order three and the others consisting of elements of order $3p$) 
    and $p$ classes of size three (one consisting of elements of order two and the others consisting of elements of
    order $2p$).  Applying \ref{maxsg} to the 
    unique subgroup of index two in $G$, we see that $\srl(G)$ has a maximal chain of length $4p$.  
    Applying \ref{maxsg} to the first component $S_3$ in the direct product, we see that 
    $\srl(G)$ has a maximal chain of length $$4+3p-3=3p+1<4p.$$

    \item 
    Assume that $G=D_8 \times \ZZ_p$ or $G=Q_8 \times \ZZ_p$.  Then $G$ has $2p$ central conjugacy classes and $3p$ 
    classes of size two.  Applying \ref{maxsg} to a subgroup of index two in $G$ (which is abelian), we see that 
    $\srl(G)$ has a maximal chain of length $6p$.  Applying \ref{maxsg} to a non-abelian subgroup of index 
    $p$ in $G$ (which is normal), we see that $\srl(G)$ has maximal chain of length  $$6+5p-5=5p+1<6p.$$

    \item 
    Assume that $G=D_{2k}$ with $k>3$ odd. If some prime $p>3$ divides $k$, then $G$ contains a dihedral subgroup of 
    order $2p$, which is minimal non-abelian and not isomorphic to any of $S_3$, $D_8$ or $Q_8$.  Thus $\srl(G)$ 
    is not graded by \ref{mns}.  It remains to consider the case where $k$ is a power of three, and it suffices to 
    show that $\srl(G)$ is not graded when $k=9$.  Applying \ref{maxsg} to the unique subgroup of index two in $G=D_{18}$, 
    we see that $\srl(G)$ has a maximal chain of length ten.  Applying \ref{maxsg} to any subgroup isomorphic to 
    $S_3$ (which is not normal), we see that $\srl(G)$ has a maximal chain of length $4+1+3=8$.

    \item
    Assume $G=TV$ with $T= \langle t \rangle \cong \ZZ_2$, $\ZZ_3 \oplus \ZZ_3 \cong V \lhd G$ and 
    $t^{-1}vt=v^{-1}$ for all $ \in V$.  Then $G$ has one central conjugacy class, one class of size nine and four 
    classes of size two.  Applying \ref{maxsg} to $V$, we see that $\srl(G)$ has a maximal chain of length ten.  
    Applying \ref{maxsg} to any subgroup of order six, we see that $\srl(G)$ has a maximal chain of length $4+1+3=8$.

    \item Assume $G=\SL_2(3)$.  Then $G$ has two central conjugacy classes, one class of size six (consisting of elements of order 
    four), and four classes of size four (two consisting of elements of order three and two consisting of elements of order six).  
    Applying \ref{maxsg} to the normal Sylow $2$-subgroup of $G$ (which is isomorphic with $Q_8$), we see that $\srl(G)$ has 
    a maximal chain of length $6+4=10$.  Applying \ref{maxsg} to any cyclic subgroup of order six (which is maximal and 
    not normal and intersects all four classes of size four non-trivially), we see that $\srl(G)$ has a maximal chain of length $6+1+1=8$.
    \end{enumerate}
  \end{proof}

  \begin{Proposition} \label{ind2}
    Let $G$ be a non-abelian finite group with an abelian subgroup $A<G$ of index two.  If $\srl(G)$ is graded, then $G$ is isomorphic 
    to one of $S_3$, $D_8$ or $Q_8$.
  \end{Proposition}

  \begin{proof}
    Assume $\srl(G)$ is graded. 

    As $G$ is non-abelian, 
    there is some minimal non-abelian subgroup $H$ of $G$.  By \ref{mns}, $H$ is isomorphic to one of $S_3$, $D_8$ or $Q_8$.  
    Moreover, $[H:H \cap A]=2$. In particular, if $H = S_3$ then $A$ contains all elements of order $3$ in $G$.  
    Fix some $t \in H \setminus A$. Note that $t$ is of order $2$ or $4$. 

    Let $B$ be a Hall $2^\prime$-subgroup of $A$.  So, $B$ consists of all elements of odd order in $A$.  As $B$ is characteristic 
    in $A$, we see that $t$ normalizes $B$.  Note that $C_B(t)=C_B(H)$, as $H=\langle t,H \cap A \rangle$ and $B \leq H \cap A$.  
    Note also that $C_B(H) \cap H=1$, as $Z(H)$ contains no non-trivial element of odd order. It follows now that if $C_B(t)$ is non-trivial, 
    then there is some element of prime order $p$ in $C_B(H) \setminus H$.  In this case, $G$ contains a subgroup isomorphic to 
    $H \times \ZZ_p$, which is impossible by \ref{small}.

    It follows that $C_B(t)=1$.  As $t^2 \in A$, we know that conjugation by $t$ induces an automorphism of $B$ having order two.  
    As $B$ is abelian, this means that 
    \begin{eqnarray*}
      t^{-1} (x(t^{-1}xt)) t & = & (t^{-1}xt) (t^{-1}t^{-1} x tt) \\
                          & = & t^{-1}xt x \\
                          & = & x t^{-1}xt 
    \end{eqnarray*}
    for each $x \in B$.  From $C_B(t) = 1$ we then deduce that $t^{-1}xt=x^{-1}$ for all 
    $x \in B$. If $B \neq 1$ and $t$ is of order $4$ then $G$ contains a minimal non-abelian subgroup not listed in 
    \ref{mns}. If $|B|=1$ then $G$ is a $2$-group and therefore isomorphic with $D_8$ or $Q_8$ by \ref{nab2}.
    Thus $B \neq 1$ and $t$ is of order $2$. By \ref{small}(3,4) it follows that $|B| = 3$ and $\langle t,B \rangle \cong S_3$. 
    Let $P$ be a Sylow $2$-subgroup of $A$.  So, $A=P \oplus B$.  
    If $P \neq 1$ then $P$ contains an odd number of elements of order two (as these elements generate a non-trivial elementary abelian group).  
    It follows that there is some $x \in C_P(t)$ of order $2$.  But then $\langle x,t,B \rangle \cong S_3 \times \ZZ_2$, 
    contradicting \ref{small}. Therefore, $P=B$ and $G \cong S_3$.
  \end{proof}

  We are now ready to prove \ref{graded}.  

  \begin{proof}[Proof of \ref{graded}] 
    Assume that $\srl(G)$ is graded and that $G$ is non-abelian.  By \ref{solvable}, $G$ is solvable.  
    Fix a composition series $$1=G_0 \lhd G_1 \lhd \ldots \lhd G_r=G.$$  Let $m$ be the smallest $j$ such that $G_j$ is non-abelian and let 
    $H$ be a minimal non-abelian subgroup of $G_m$.  

    \noindent {\sf Claim:} $H=G_m=G$.

    $\triangleleft$ 
    As $G_{m-1}$ is abelian and $H$ is generated by elements of $2$-power order (by \ref{mns}), we see that some element of $2$-power order 
    lies in $H \setminus G_{m-1}$.  As $[G_m:G_{m-1}]$ is prime, we see that $[G_m:G_{m-1}]=2$.  
    From the fact that $G_{m-1}$ is abelian and \ref{ind2} it follows that $G_m=H$.
    $\triangleright$ 

    It remains to show that $G_m=G$. Assume for contradiction that this is not the case. 
    Then $r \geq m+1$ and $[G_{m+1}:G_m]=p$ for some prime $p$.  

    Since $G_m$ is minimal non-abelian by \ref{mns} we have to distinguish the following cases.

    \smallskip

    \noindent {\sf Case:} $G_m \cong S_3$

    It is well known and not hard to show that every automorphism of $S_3$ is inner.  It follows 
    that we can choose some $x \in G_{m+1}$ such that $\langle G_m,x \rangle=G_{m+1}$ and $x$ centralizes $G_m$.  As $S_3$ has trivial center, 
    we see that $\langle x \rangle \cap G_m=1$ and $G_m \cong S_3 \times \ZZ_p$, contradicting \ref{small}.

    \smallskip

    \noindent {\sf Case:} $G_m \cong D_8$ or $G_m \cong Q_8$

    By \ref{nab2}, $G_m$ is a Sylow $2$-subgroup of $G$.  Therefore, 
    $G_{m+1}$ is the semi-direct product $PG_m$, where $P$ is cyclic of odd prime order $p$.  We claim that either this product is direct, 
    or $G_{m+1} \cong \SL_2(3)$.  \ref{graded} follows from the claim and \ref{small}.  

   The claim is in fact known to be true, and is not hard to prove.  We give a sketch here.  

   An odd order automorphism of a $2$-group is determined by its action on the Frattini quotient.  
   In both cases the Frattini quotient is elementary abelian of rank two. Since $\GL_2(2)$ has order six, 
   the only non-trivial action is obtained by taking $P$ to be the subgroup of order $3$ in $\GL_2(2)$.
   
   Thus, if $G_m \cong Q_8$, then either $PG_m$ is a direct product or $G_{m+1} \cong SL_2(3)$.  
   Consider the case $G_m = D_8$. Indeed $D_8$ has no automorphism of order three. Note, since $D_8$ has three maximal subgroups, 
   not all isomorphic, any automorphism of order three must normalize all maximal subgroups. From this and that fact that all maximal
   subgroups contain the center of $D_8$, it quickly follows that any automorphism of order three centralizes each maximal subgroup and 
   therefore centralizes all of $D_8$.
  \end{proof}

\section{Questions and Problems}
  \label{sec5}

%  \subsection{Universality for subrack lattices}

  If $R$ is a quandle then $\srl(R)$ is easily seen to be an atomic lattice. 
  For general (finite) racks we do not know if this property holds.

  \begin{Question}
    Is $\srl(R)$ atomic for all racks $R$ ?
  \end{Question}

  Consider the converse question for quandles: Does any atomic lattice appear 
  as the subrack lattice of a quandle ? This question has a negative answer.
  In \cite{Ve} it was conjectured and later proved in \cite{Mc} and
  then in \cite{HuStVo} that there is no indecomposable quandle of size
  $2p$ for $p > 5$. In particular, this implies that the lattice
  with one top and one bottom element and an antichain of $2p$ 
  incomparable elements in between is not the subrack lattice of a
  quandle. 
  However, is not clear which general restrictions apply to subrack lattices of
  quandles.

  We have seen in \ref{sec2} that for a rack $R$ the distribution of the 
  numbers $i$ for which 
  $\widetilde{H}_i(\Delta(\srl(R)),\ZZ) \neq 0$ can be rather complicated.
  On the other hand for all racks $R$ we have studied 
  $\widetilde{H}_\bullet(\Delta(\srl(R)),\ZZ) \neq 0$ and is torsion free. 
  Indeed, we do not know of any rack $R$ such that $\Delta(\srl(R))$ is
  acyclic. Since the order complex of a lattice is acyclic if 
  the lattice is not complemented, we also do not know of any
  rack $R$ for which $\srl(R)$ is not complemented. 

  \begin{Question} Is there a rack $R$ such that
    $\srl(R)$ is not complemented ?
    Is there a rack $R$ such that 
    $\Delta(\srl(R))$ is acyclic ? Is there a rack $R$ such that 
    $\Delta(\srl(R))$ has torsion in homology ?
  \end{Question}
  \ref{main} also raises the question how much more about the group
  structure can be deduced from the subrack lattice. For 
  $G = D_8$ and $H = Q_8$ the two non-abelian groups of order $8$ 
  we have $\srl(G) \cong \srl(H)$. In \ref{fig1} we depict the 
  subrack lattice of $\srl(R)$ where $R$ is the rack of non-central 
  elements of $G$ or $H$. By \ref{le:product}(1) $\srl(G) \cong \srl(H)
  \cong \srl(R) \times 2^{[2]}$.  
  But in this case it can be checked that the racks $G$ and $H$ are 
  actually isomorphic. 

  \begin{Question}
    Are there two groups $G,H$, which have isomorphic subrack lattices,
    but are non-isomorphic as racks?
  \end{Question}

  \begin{figure}
      \setlength{\unitlength}{0.75cm}
       \vskip9cm
       \begin{picture}(0,0)(5,2)
         \includegraphics[width=0.6\textwidth]{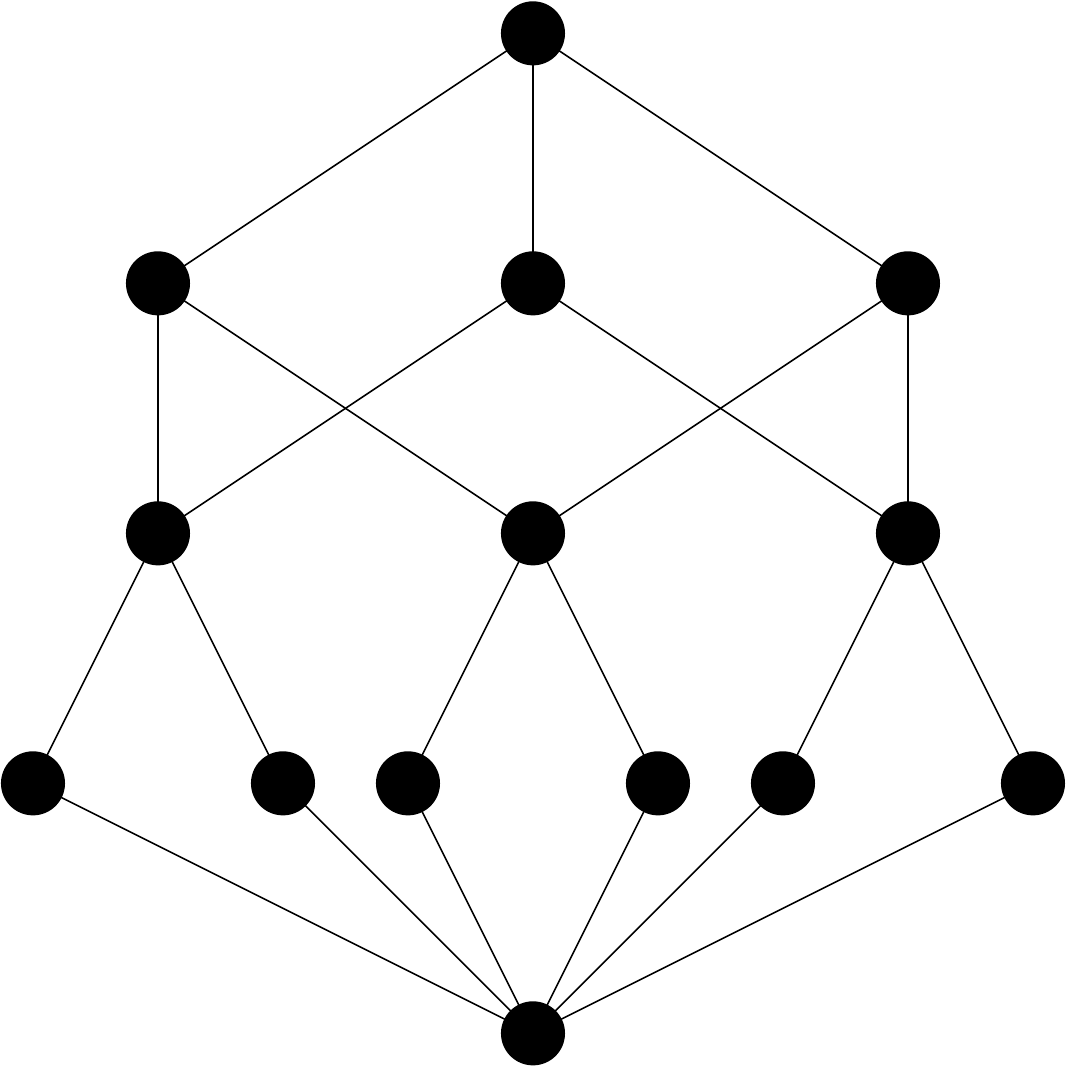}
       \end{picture}%
       \vskip1.3cm
       \caption{Subracklattice of non-central elements of $D_8 $ or $Q_8$}
      \label{fig1}
  \end{figure}

\end{document}